\documentclass[11pt]{article}
\usepackage{amsthm}
\usepackage{amsmath}
\usepackage{amssymb}
\usepackage{latexsym}
\usepackage{amsfonts}
\usepackage{placeins}
\usepackage{caption}
\usepackage{mathrsfs}
\usepackage[latin1]{inputenc}
\usepackage{listings}
\newtheorem{theorem}{\bf Theorem}[section]

\newtheorem{lemma}[theorem]{\bf Lemma}
\begin{document}
\title{$k$-Pell-Lucas numbers as Product of Two Repdigits}
\author{Bibhu Prasad Tripathy and Bijan Kumar Patel}
\date{}
\maketitle
\begin{abstract}
For any integer $k \geq 2$, let $\{Q_{n}^{(k)} \}_{n \geq -(k-2)}$ denote the $k$-generalized Pell-Lucas sequence which starts with $0, \dots ,2,2$($k$ terms) where each next term is the sum of the $k$ preceding terms. In this paper, we find all the $k$-generalized Pell-Lucas numbers that are the product of two repdigits. This generalizes a result of Erduvan and Keskin \cite{Erduvan1} regarding repdigits of Pell-Lucas numbers.
\end{abstract} 


\section{Introduction}
The Pell-Lucas sequence $\{Q_n\}_{n\geq 0}$ is the binary recurrence sequence defined by 
\[
Q_{n+2} = 2 Q_{n+1} + Q_{n}~~ {\rm for} ~n\geq 0
\]
with initials $Q_0 = 2$ and $Q_1 = 2$. 

Let $ k \geq 2 $ be an integer. We consider the $k$-generalized Pell-Lucas sequence $\{Q_{n}^{(k)} \}_{n \geq -(k-2)}$, which is a generalization of the Pell-Lucas sequence and is given by the recurrence
\begin{equation}\label{eq 1.1}
  Q_{n}^{(k)} = 2 Q_{n-1}^{(k)} + Q_{n-2}^{(k)} + \dots + Q_{n-k}^{(k)} = \sum_{i=1}^{k} Q_{n-i}^{(k)} ~~\text{for all}~n \geq 2,  
\end{equation}
with initials $Q_{-(k-2)}^{(k)} = Q_{-(k-3)}^{(k)} = \dots = Q_{-1}^{(k)} = 0,  Q_{0}^{(k)} = 2$ and $Q_{1}^{(k)} = 2$.

We shall refer to $Q_{n}^{(k)}$  as the $k$-Pell-Lucas sequence.  We see that this generalization is a family of sequences, with each new choice of $k$ producing a unique sequence. For example, if $k = 2$, we get $Q_{n}^{(2)} = Q_{n}$, the classical Pell-Lucas sequence. For $k = 3$, we have the  3-Pell-Lucas sequence. They are followed by the 4-Pell-Lucas sequence for $k = 4$, and so on.

A positive integer of the form $\frac{a(10^{l} - 1)}{9}$, for any $l \geq 1$ and $1 \leq a \leq 9$, is said to be a repdigit if its decimal expansion only contains one distinct digit. There are many literature in number theory concerning Diophantine equations involving binary recurrence sequences of positive integers and repdigits. First Faye and Luca \cite{Faye} found that $Q_{1}= 2 $ and $Q_{2}= 6$ are the only repdigits in the Pell-Lucas sequence. Later  Adegbindin et al.  \cite{Adegbindin} showed that the largest Pell-Lucas number which is a sum of two repdigits is $Q_{6} = 198 = 99 + 99$. Consequently, Adegbindin et al. \cite{Adegbindin1} found that the largest Pell-Lucas number which is a sum of three repdigits is $Q_{10} = 6726 = 6666 + 55 + 5$. Following that Erduvan and Keskin \cite{Erduvan1} showed  that $Q_{6} = 198$ is the largest Pell-Lucas number which is a product of two repdigits. 

Recently, Rihane \cite{Rihane} extended the previous work of \cite{Erduvan} and showed that $F_{8}^{(3)}
 = 4 \cdot 11$ and $F_{8}^{(3)}
 = 22 \cdot 88 = 44^{2}$ are the only $k$-Fibonacci number $F_{n}^{(k)}$ which are product of two repdigits. He also found in the same paper  that $L_{9}^{(4)} = 7 \cdot 44 = 4 \cdot 77, L_{10}^{(4)} = 9 \cdot 66 = 6 \cdot 99$ and $ L_{9}^{(5)} = 8 \cdot 44 = 4 \cdot 88$ are the only $k$-Lucas number $L_{n}^{(k)}$ that can be expressed as product of two repdigits. Consequently, he \cite{Rihane} also extended the previous work of \cite{Erduvan1} and studied it in $k$-Pell sequence.

Motivated by the above literature, we solve the Diophantine equation 
\begin{equation}\label{eq 1.2}
  Q_{n}^{(k)} = \frac{a(10^{l}-1)}{9} \cdot \frac{b(10^{m}-1)}{9}.
\end{equation}
More precisely, we find all the $k$-Pell-Lucas numbers that are  product of two repdigits. In particular, we have the following result.
\begin{theorem}\label{thm1}
  All the solutions of the Diophantine equation \eqref{eq 1.2} in positive integers $n, m, l , k , a$ and $b$  with $k \geq 3$, $1 \leq l \leq m, 1 \leq a \leq 9$  and $1 \leq b \leq 9$ are
\begin{align*}
& Q_{1}^{(k)} = 1 \cdot 2  = 2 \cdot 1, k \geq 3, & Q_{2}^{(k)} = 1 \cdot 6  = 2 \cdot 3 = 3 \cdot 2 = 6 \cdot 1 , k \geq 3, \\
& Q_{3}^{(k)} = 2 \cdot 8  = 4 \cdot 4 = 8 \cdot 1, k \geq 3, & Q_{4}^{(3)} = 5 \cdot 8  = 8 \cdot 5, Q_{4}^{(k)} = 6 \cdot 7  = 7 \cdot 6, k \geq  4,  \\
& Q_{5}^{(k)} = 22 \cdot 5  = 55 \cdot 2 = 2 \cdot 55 = 5 \cdot 22, k \geq 3, & Q_{7}^{(4)} = 11 \cdot 66  = 66 \cdot 11 = 22 \cdot 33 = 33 \cdot 22.
\end{align*} 
\end{theorem}

Our proof of Theorem \ref{thm1} is mainly based on linear forms in logarithms of algebraic numbers and a reduction algorithm originally introduced by Baker and Davenport \cite{Baker} (and improved by Dujella and Peth\H{o} \cite{Dujella}). Here, we also use a variant due to de Weger \cite{Weger} to reduce the upper bound to a size that is simpler to handle for the small case. Here we have also derived some estimates for large case by using the fact that the dominant root of the $k$-Pell-Lucas sequence is exponentially close to $\phi^2$ {\rm{(see \cite{Siar})}} where $\phi = \frac{1 + \sqrt{5}}{2}$. Consequently, one may substitute this root by $\phi$ in further calculation with linear forms in logarithms to obtain absolute upper bounds for all the variables, which we then reduce using again the de Weger algorithm.

\section{Auxiliary results}
 In this section, we recall the concept of a lower bound for a non-zero linear forms in logarithms of algebraic
numbers. We also discuss some fundamental properties and results of $k$-Pell-Lucas sequence.

\subsection{Linear forms in logarithms}
Let $\gamma$ be an algebraic number of degree $d$ with minimal primitive polynomial 
\[
f(Y):= b_0 Y^d+b_1 Y^{d-1}+ \cdots +b_d = b_0 \prod_{j=1}^{d}(Y- \gamma^{(j)}) \in \mathbb{Z}[Y],
\]
where the $b_j$'s are relatively prime integers, $b_0 >0$, and the $\gamma^{(j)}$'s are conjugates of $\gamma$. Then the \emph{logarithmic height} of $\gamma$ is given by
\begin{equation}\label{eq 2.3}
h(\gamma)=\frac{1}{d}\left(\log b_0+\sum_{j=1}^{d}\log\left(\max\{|\gamma^{(j)}|,1\}\right)\right).
\end{equation}
With the above notations, Matveev (see  \cite{Matveev} or  \cite[Theorem~9.4]{Bugeaud}) proved the following result.

\begin{theorem}\label{thm2}
Let $\eta_1, \ldots, \eta_s$ be positive real algebraic numbers in a real algebraic number field $\mathbb{L}$ of degree $d_{\mathbb{L}}$. Let $a_1, \ldots, a_s$ be non-zero  integers such that
\[
\Lambda :=\eta_1^{a_1}\cdots\eta_s^{a_s}-1 \neq 0.
\]
Then
\[
- \log  |\Lambda| \leq 1.4\cdot 30^{s+3}\cdot s^{4.5}\cdot d_{\mathbb{L}}^2(1+\log d_{\mathbb{L}})(1+\log D)\cdot B_1 \cdots B_s,
\]
where
\[
D\geq \max\{|a_1|,\ldots,|a_s|\},
\]
and
\[
B_j\geq \max\{d_{\mathbb{L}}h(\eta_j),|\log \eta_j|, 0.16\}, ~ \text{for all} ~ j=1,\ldots,s.
\]
\end{theorem}
The following lemma is a slight variation of a result due to Dujella and Peth\"{o} \cite[Lemma~5 (a)]{Dujella} and a generalisation of a Baker and Davenport result \cite{Baker}.

\subsection{The de Weger reduction algorithm}
Here, we present a variant of the reduction method of Baker and Davenport \cite{Baker}
(and improved by Dujella and Peth\H{o} \cite{Dujella}) due to de Weger \cite{Weger}.

Let $ \vartheta_{1}, \vartheta_{2}, \delta \in \mathbb{R}$ be given and let $x_{1}, x_{2} \in \mathbb{Z}$ be unknowns. Let 
\begin{equation} \label{eq 2.4}
    \Lambda = \delta + x_{1}\vartheta_{1} + x_{2}\vartheta_{2}. 
\end{equation}
Set $X = \max \{|x_{1},|x_{2}|\}$. Let $X_{0}, Y$ be positive. Assume that 
\begin{equation} \label{eq 2.5}
    |\Lambda| < c \cdot \exp{(-\rho \cdot Y)}
\end{equation}
and 
\begin{equation} \label{eq 2.6}
    Y \leq X \leq X_{0},
\end{equation}
where $c, \rho$ be positive constants.
\noindent
When $\delta = 0$ in \eqref{eq 2.4}, we get 
\[
\Lambda = x_{1}\vartheta_{1} + x_{2}\vartheta_{2}.
\]
Put $\vartheta = - \vartheta_{1}/\vartheta_{2}$. We assume that 
$x_{1}$ and $x_{2}$ are coprime. Let the continued fraction expansion of $\vartheta$ be given by 
\[
[a_{0}, a_{1}, a_{2}, \dots],
\]
and let the $k$th convergent of $\vartheta$ be $p_{k}/q_{k}$ for $k=0,1,2,\dots$. We may assume without loss of generality that $|\vartheta_{1}| < |\vartheta_{2}|$ and that $x_{1} > 0$. We have the following results.

\begin{lemma}\label{lem 2.2}
(\cite{Weger}, Lemma 3.1) If \eqref{eq 2.5} and \eqref{eq 2.6} hold for $x_{1}$, $x_{2}$ with $X \geq 1$ and $\delta = 0$, then $(-x_{2}, x_{1}) = (p_{k},q_{k})$ for an index $k$ that satisfies 
\[
k \leq -1 + \frac{\log(1+X_{0}\sqrt{5})}{\log\left( \frac{1+\sqrt{5}}{2}\right)} := Y_{0}.
\]
\end{lemma}
\begin{lemma}\label{lem 2.3}
(\cite{Weger}, Lemma 3.2) Let 
\[
A = \displaystyle\max_{0 \leq k \leq Y_{0}} a_{k+1}.
\]
If \eqref{eq 2.5} and \eqref{eq 2.6} hold for $x_{1}$, $x_{2}$ with $X \geq 1$ and $\delta = 0$, then 
\begin{equation}\label{eq 2.7}
Y < \frac{1}{\rho} \log \left( \frac{c(A+2)}{|\vartheta_{2}|} \right) + \frac{1}{\rho} \log X  < \frac{1}{\rho} \log \left( \frac{c(A+2)X_{0}}{|\vartheta_{2}|} \right).    
\end{equation}
\end{lemma}
\noindent
When $\delta \neq 0$ in \eqref{eq 2.4}, put $\vartheta = - \vartheta_{1}/\vartheta_{2}$ and $\psi = \delta / \vartheta_{2}$. Then we have \[
\frac{\Lambda}{\vartheta_{2}} = \psi - x_{1}\vartheta + x_{2}.
\]
Let $p/q$ be a convergent of $\vartheta$ with $q > X_{0}$. For a real number $x$ we let $
\|x\| = \min \{|x-n|: n\in \mathbb{Z} \}$ be the distance from $x$ to the nearest integer. We have the following result.

\begin{lemma}\label{lem 2.4}
(\cite{Weger}, Lemma 3.3) Suppose that 
\[
\|q \psi \| > \frac{2X_{0}}{q}.
\]
Then, the solutions of \eqref{eq 2.5} and \eqref{eq 2.6} satisfy 
\[
Y < \frac{1}{\rho} \log \left( \frac{q^{2}c}{|\vartheta_{2}|X_{0}} \right).
\]
\end{lemma}
\begin{lemma}\label{lem 2.5} (\cite{Weger})
Let $a, x \in \mathbb{R}$. If $0< a < 1$ and $|x| < a$, then 
\[
|\log(1+x)| < \frac{-\log(1-a)}{a}\cdot |x|
\]
and 
\[
|x| < \frac{a}{1 - e^{-a}} \cdot |e^{x} - 1|.
\]
\end{lemma}
\begin{lemma}\label{lem 2.6}
(\cite{Sanchez}, Lemma 7)
If $m \geq 1$, $S \geq (4m^{2})^{m}$ and $\frac{x}{(\log x)^{m}} < S$, then $x < 2^{m} S (\log S)^{m}$.
\end{lemma}

\subsection{Properties of $k$-Pell-Lucas sequence}
The characteristic polynomial of the $k$-Pell-Lucas sequence is  
 \[
 \Phi_{k}(x) = x^{k} - 2 x^{k-1}  - x^{k-2} - \dots - x -1. 
 \]
The above polynomial is irreducible over $\mathbb{Q} [x]$ and it has one positive real root that is $\gamma := \gamma(k)$ which is located between $ \phi^ 2(1 - \phi ^{-k})$ and $\phi^2$, lies outside the unit circle (see \cite{Siar}). The other roots are firmly contained within the unit circle. To simplify the notation, we will omit the dependence on $k$ of $\gamma$ whenever no confusion may arise.

The Binet  formula for $Q_{n}^{(k)}$  that found in \cite{Siar} is
\begin{equation}\label{eq 2.8}
Q_{n}^{(k)} = \displaystyle\sum_{i=1}^{k} (2 \gamma_{i} - 2) g_{k} (\gamma_{i})\gamma_{i}^{n} = \sum_{i=1}^{k} \frac{2(\gamma_{i}-1)^{2}}{(k+1)\gamma_{i}^2 - 3k\gamma_{i} + k -1} \gamma_{i}^{n},
\end{equation}
where  $\gamma_{i}$ represents the roots of the characteristic polynomial $\Phi_{k}(x)$ and the function $g_{k}$ is  given by
\begin{equation}\label{eq 2.9}
	g_{k}(z) := \frac{z-1}{(k+1)z^2 - 3kz + k -1}, 
\end{equation}
for an integer $k$ $\geq$ 2. 

Additionally it is  also shown in \cite[Lemma~10]{Siar} that the roots located inside the unit circle have a very minimal influence on the formula \eqref{eq 2.9}, which is given by the approximation 
\begin{equation}\label{eq 2.13}
	\left| Q_{n}^{(k)} -  (2\gamma-2)g_{k} (\gamma) \gamma^{n} \right| < 2 \quad \text{holds  for   all} \quad n \geq 2 - k.
\end{equation}
Furthermore, it was shown by \c{S}iar and Keskin in \cite[Lemma~10]{Siar} that the inequality
\begin{equation}\label{eq 2.14}
	\gamma^{n-1} \leq Q_{n} ^ {(k)} \leq 2 \gamma^{n} \text{ holds for all } n \geq 1 ~{\rm and}~ k \geq 2.
\end{equation}
\begin{lemma}\label{lem 2.7}{\rm{(\cite{Bravo}, Lemma 3.2)}}.
Let $k \geq 2 $  be an integer. Then we have
\[
0.276 < g_{k}(\gamma) < 0.5 \  and \  \left| g_{k}(\gamma_{i}) \right| < 1 \quad for \quad 2 \leq i \leq k.
\]	
\end{lemma}
\noindent
Furthermore, Bravo et al. \cite{Bravo1} showed that the logarithmic height of $g_{k} (\gamma)$ satisfies
\begin{equation}\label{eq 2.15}
    h(g_{k}(\gamma)) < 4k\log \phi + k \log(k+1) \quad \text{for all} \quad k \geq 2.
\end{equation}
\begin{lemma}\label{lem 2.8}
Let $\gamma$ be the dominant root of the characteristic polynomial $\Phi_{k}(x)$  and consider the function $g_{k}(x)$ defined in \eqref{eq 2.9}. If $k \geq 50$ and $n > 1$ are integers satisfying $n < \phi^{k/2}$, then the following inequalities holds
\\
(i) \cite[Equation~30]{Siar}
\[
\left| (2 \gamma - 2)\gamma^{n} - 2 \phi^{2n+1} \right| < \frac{4 \phi^{2n}}{\phi^{k/2}},
\]
(ii)\cite[Lemma~13]{Siar}
\[
|g_{k}(\gamma) -  g_{k}(\phi^{2})| < \frac{4k}{\phi^{k}}.
\]
\end{lemma}
\begin{lemma}\label{lem 2.9}
Let $k \geq 50$ and suppose that $n < \phi^{k/2}$, then
\begin{equation}\label{eq 2.16}
   (2 \gamma - 2)g_{k} (\gamma) \gamma^{n} = \frac{2 \phi^{2n+1}}{\phi + 2}(1 + \xi), \quad \text{where} \quad  |\xi| < \frac{1.25}{\phi^{k/2}}.
\end{equation}
\end{lemma}
\begin{proof}
By virtue of Lemma \ref{lem 2.8}, we have
\begin{equation}\label{eq 2.21}
(2 \gamma - 2)\gamma^{n} = 2 \phi^{2n+1} + \delta \quad \text{and} \quad g_{k}(\gamma) = g_{k}(\phi^{2}) + \eta, 
\end{equation}
where
\begin{equation}\label{eq 2.22}
 |\delta| = \frac{4 \phi^{2n}}{\phi^{k/2}} \quad \text{and} \quad |\eta| = \frac{4k}{\phi^{k}}.   
\end{equation}
Since $g_{k}(\phi^{2}) = \frac{1}{\phi + 2}$, we can write
\begin{align*}
(2 \gamma - 2)\gamma^{n}g_{k}(\gamma) & = \left( 2 \phi^{2n+1} + \delta \right)  \left( g_{k}(\phi^{2}) + \eta\right) \\ 
&  = \frac{2 \phi^{2n+1}}{\phi+2}(1 + \xi),
\end{align*}
where 
\[
\xi = \frac{\delta}{2 \phi^{2n+1}} + (\phi +2)\eta + \frac{(\phi +2)\eta \delta}{2 \phi^{2n+1}}.
\]
Using \eqref{eq 2.22}, it follows from the  above equation that
\[
|\xi|  < \frac{2/\phi}{\phi^{k/2}} + \frac{4k(\phi+2)}{\phi^{k}} + \frac{8k(\phi+2)}{\phi \cdot \phi^{3k/2}} < \frac{1.25}{\phi^{k/2}},
\]
where we have used the facts 
\[
\frac{2/\phi}{\phi^{k/2}} < \frac{1.24}{\phi^{k/2}}, \quad  \frac{4k(\phi+2)}{\phi^{k}} < \frac{0.005}{\phi^{k/2}} \quad \text{and} \quad \frac{(8k/\phi)(\phi+2)}{\phi^{3k/2}} < \frac{0.005}{\phi^{k/2}}, 
\]
for all  $k \geq 50$. This completes the proof of the lemma.
\end{proof}

\section{Proof of Theorem \ref{thm1}}
First, observe that for $1 \leq n \leq k+1$,  we have  $Q_{n}^{(k)} = 2 F_{2n}$ (see $\cite[Lemma~10]{Siar}$ ) where $F_{m}$ is the $m$th Fibonacci number. Therefore \eqref{eq 1.2} becomes 
\begin{equation}\label{eq 3.23}
  2 F_{2n} =  \frac{a(10^{l}-1)}{9} \cdot \frac{b(10^{m}-1)}{9}.
\end{equation}
So by using \cite[Theorem~4]{Erduvan}, we found that the solutions of  \eqref{eq 3.23} satisfy for $n \in \{ 1,2,3,4\}$. From now, we consider $n \geq k+2$ and $k \geq 3$. 

\subsection{An initial relation between $n$ and $m+l$ }
The inequalities \eqref{eq 2.14} and the relations $10^{l - 1} < a(10^{l} - 1)/9 < 10^{l}$ implies that 
\[
10^{m+l-2} < \frac{a(10^{l}-1)}{9} \cdot \frac{b(10^{m}-1)}{9} = Q_{n}^{(k)} < 2 \gamma^{n}
\]
and 
\[
\gamma^{n-1} < Q_{n}^{(k)} = \frac{a(10^{l}-1)}{9} \cdot \frac{b(10^{m}-1)}{9} = 10^{m+l}. 
\]
This shows that 
\[
(n-1) \left( \frac{\log \gamma}{\log 10}  \right) < m+l < \frac{\log 2}{\log 10} + n \left( \frac{\log \gamma}{\log 10} \right) + 2.
\]
Since $\phi^{2}(1-\phi^{-k}) < \gamma(k) < \phi^{2}$ for all $k \geq 3$, we obtain
\begin{equation}\label{eq 3.24}
0.3n - 0.3 < m+l < 0.42n + 2.31.    
\end{equation}
\subsection{Upper bounds for $n$ in terms of $k$}
Now, we can rearrange the equation \eqref{eq 1.2} as 
\begin{equation}\label{eq 3.25}
    \frac{ab10^{m+l}}{81} = Q_{n}^{(k)} + \frac{ab}{81}(10^{l} + 10^{m} -1).
\end{equation}
Using \eqref{eq 2.13} and taking absolute values on both sides, we obtain 
\begin{align*}
\left |\frac{ab10^{m+l}}{81} - (2\gamma - 2)g_{k}(\gamma) \gamma^{n}  \right| &= \left |\frac{ab}{81} (10^{l} + 10^{m} -1) + Q_{n}^{(k)}  - (2\gamma - 2)g_{k}(\gamma) \gamma^{n}  \right|  \\ 
& \leq  \frac{ab}{81} (10^{l} + 10^{m} + 1) + 2. 
\end{align*}
Dividing both sides by $\frac{ab10^{m+l}}{81}$, we get 
\begin{equation} \label{eq 3.26}
|\Lambda_{1}| \leq \frac{162}{ab10^{m+l}} + \frac{1}{10^{m}} + \frac{1}{10^{l}} + \frac{1}{10^{m+l}} < \frac{18.3}{10^{l}},
\end{equation}
where 
\begin{equation} \label{eq 3.27}
    \Lambda_{1} := \frac{81(2\gamma - 2)g_{k}(\gamma)}{ab}  \gamma^{n} 10^{-(m+l)} - 1.
\end{equation}
If $\Lambda_{1} = 0$, then we would get 
\[
(2\gamma - 2)g_{k}(\gamma) = \frac{ab}{81} \cdot \gamma^{-n} \cdot 10^{m+l}
\]
which implies that $g_{k}(\gamma)$ is an algebraic integer, which is a contradiction. Hence $\Lambda_{1} \neq 0$.

\noindent
In order to use Theorem \ref{thm2} to $\Lambda_{1}$ given by \eqref{eq 3.27}, we take the parameters:
\[
\eta_{1} := \frac{81(2 \gamma-2)g_{k}(\gamma)}{ab}, \quad \eta_{2} := \gamma , \quad \eta_{3} := 10,
\]
and
\[ a_{1}:= 1, \quad a_{2}:= n, \quad a_{3}:= -(m+l).
\]
Note that $\eta_{1}, \eta_{2}, \eta_{3}$ belongs to the field  $\mathbb{L} := \mathbb{Q}(\gamma)$, so we can assume $d_{\mathbb{L}} = [\mathbb{L}:\mathbb{Q}] \leq k$. Since $h(\eta_{2}) = (\log \gamma) / k < (2 \log \phi)/k$ and $h(\eta_{3}) = \log 10 $, it follows that 
\[
\max\{kh(\eta_{2}),|\log \eta_{2}|,0.16\} = 2 \log \phi := B_{2}
\]
and 
\[
\max\{kh(\eta_{3}),|\log \eta_{3}|,0.16\} = k\log 10 := B_{3}.
\]
Therefore, by \eqref{eq 2.15} and the properties of logarithmic height, it follows that
\begin{align*}
    h(\eta_{1}) & = h\left( \frac{81 (2\gamma - 2) g_{k}(\gamma)}{ab}\right)\\ 
    & \leq h(81) + h(2\gamma - 2) + h(g_{k}(\gamma)) + h(a) + h(b) \\
    &< 4\log 9 + 4 k \log \phi + k \log (k+1) + \log 4 \\
    &< 6.2 k \log k, \text{ for all} \quad k \geq 3.
\end{align*}
Since $0.276 < g_{k}(\gamma) < 0.5$, we have
\[
\eta_{1} = \frac{81 (2\gamma - 2) g_{k}(\gamma)}{ab} < 162 \quad \text{and} \quad \eta_{1}^{-1} = \frac{ab}{81 (2\gamma - 2) g_{k}(\gamma)} < 1.9.
\]
Thus, we obtain
\[
\max\{kh(\eta_{1}),|\log \eta_{1}|,0.16\} = 6.2 k^2 \log k := B_{1}.
\]
In addition, using the fact $0.42n + 2.31 < n$ for $n \geq 5$, we can take $D:= n$. Then by Theorem \ref{thm2}, we have
\begin{equation}\label{eq 3.28}
- \log |\Lambda_{1}| < 1.432 \times 10^{11} k^{2} (1+ \log k) (1 +\log n) (6.2 k^{2} \log k) (2 \log \phi) (k \log 10).    
\end{equation}
From the comparison of lower bound \eqref{eq 3.28} and upper bound \eqref{eq 3.26} of   $|\Lambda_{1}|$ gives us
\[
l \log 10 - \log 18.3 < 1.97 \times 10^{12} k^{5} \log k (1 + \log k)(1+ \log n).
\]
Using the fact $1+\log k < 2 \log k$ for all $k \geq 3$, we obtain that 
\begin{equation}\label{eq 3.29}
 l \log 10 < 3.94 \times 10^{12} k^{5} \log^{2} k (1 + \log n).   
\end{equation}
Now rearranging the equation \eqref{eq 1.2} as
\begin{equation}\label{eq 3.30}
    \frac{b 10^{m}}{9} = \frac{9 Q_{n}^{(k)}}{a(10^{l} - 1)} + \frac{b}{9}.
\end{equation}
This makes it possible for us to write
\begin{align*}
\left| \frac{b 10^{m}}{9} - \frac{9(2\gamma - 2) g_{k}(\gamma)}{a(10^{l} - 1)} \cdot \gamma^{n} \right | & = \left| \frac{9 \left( Q_{n}^{(k)} - (2\gamma - 2) g_{k}(\gamma) \gamma^{n} \right)}{a(10^{l} - 1)} + \frac{b}{9}\right|  \\ 
& \leq \frac{18}{a(10^{l} - 1)} + \frac{b}{9}.
\end{align*}
Dividing both sides by $\frac{b 10^{m}}{9}$, it yields
\begin{equation}\label{eq 3.31}
 \left| \frac{81(2\gamma - 2) g_{k}(\gamma)}{ab(10^{l} - 1)} \cdot \gamma^{n} \cdot 10^{-m} - 1 \right | \leq \frac{18 \cdot 9}{ab(10^{l} - 1)10^{m}} + \frac{1}{10^{m}} < \frac{19}{10^{m}},
\end{equation}
where 
\begin{equation}\label{eq 3.32}
\Lambda_{2}:= \frac{81(2\gamma - 2) g_{k}(\gamma)}{ab(10^{l} - 1)} \cdot \gamma^{n} \cdot 10^{-m} - 1.    
\end{equation}
Suppose that $\Lambda_{2}= 0$, then we get
\[
g_{k}(\gamma) = \frac{10^{m} \gamma^{-n} ab(10^{l}-1)}{81(2\gamma - 2)}.
\]
This implies that $g_{k}(\gamma)$ is an algebraic integer, which is a contradiction. Hence $\Lambda_{2} \neq 0$. We take $s := 3$,
\[
 \eta_{1}:= \frac{81(2\gamma - 2) g_{k}(\gamma)}{ab(10^{l} - 1)} , \quad \eta_{2}:= \gamma, \quad \eta_{3}:= 10,
\]
and
\[ 
a_{1}:= 1, \quad a_{2}:= n, \quad a_{3}:= -m.
\]
Note that  $\mathbb{L} := \mathbb{Q}(\gamma)$ contains $\eta_{1}, \eta_{2}, \eta_{3}$ and has $d_{\mathbb{L}}:= k$. Since $m < n$, we deduce that $D:= \max\{|a_{1}|,|a_{2}|,|a_{3}|\}= n$. Since the logarithmic heights for $\eta_{2}$ and $\eta_{3}$ calculated as before are $h(\eta_{2}) = (\log \gamma) / k < (2 \log \phi)/k$ and $h(\eta_{3}) = \log 10 $. Therefore, we may take $ B_{2} := 2 \log \phi$ and $B_{3} := k \log 10$. Besides, using \eqref{eq 2.14} and \eqref{eq 3.28}, we have 
\begin{align*}
    h(\eta_{1}) &= h\left( \frac{81(2\gamma - 2) g_{k}(\gamma)}{ab(10^{l} - 1)} \right) \\
    & \leq h(81)+ h(g_{k}(\gamma)) + h(2\gamma - 2) + h(a) + h(b) + h(10^{l} - 1) \\
    & < 4 \log 9 + l \log 10 + 3 \log 2 + 4 k \log \phi + k \log (k+1) \\
    & < 3.95 \times 10^{12} k^{5} \log^{2} k (1+ \log n).
\end{align*}
However, given that
\[
\eta_{1} = \frac{81(2\gamma - 2) g_{k}(\gamma)}{ab(10^{l} - 1)} < 18  \quad \text{and} \quad \eta_{1}^{-1} = \frac{ab(10^{l} - 1)}{81(2\gamma - 2) g_{k}(\gamma)} < 1.9 \cdot 10^{l},
\]
so by \eqref{eq 3.29}, it follows that
\[
|\log \eta_{1}|< l \log 10 + \log 1.9 < 3.95 \times 10^{12} k^{5} \log^{2} k (1+ \log n). 
\]
As a result, we can take $B_{1} := 3.95 \times 10^{12} k^{6} \log^{2} k (1+ \log n).$ Consequently, by using Theorem \ref{thm2} with the facts $1+\log k < 2 \log k$ and $1 + \log n < 2 \log n$ for $k \geq 3$ and $n \geq\ 5$ to the inequality \eqref{eq 3.32}, we obtain that
\[
- \log |\Lambda_{2}| < 1.1 \times 10^{25} k^{9} \log^{3} k  \log^{2} n.
\]
Comparing the above inequality with \eqref{eq 3.31}, it yields
\[
m < 4.78 \times 10^{24} k^{9} \log^{3} k \log^{2} n.
\]
Using inequality \eqref{eq 3.24}, the previous inequality yields 
\begin{equation}\label{eq 3.33}
\frac{n}{\log^{2} n} < 3.2 \times 10^{25} k^{9} \log^{3} k.    
\end{equation}
Thus, putting $S := 3.2 \times 10^{25} k^{9} \log^{3} k$ in \eqref{eq 3.33} and using Lemma \ref{lem 2.6}, the above inequality yields 
\begin{align*}
n & < 4(3.2 \times 10^{25} k^{9} \log^{3} k) (\log (3.2 \times 10^{25} k^{9} \log^{3} k))^{2} \\
& < (1.28 \times 10^{26} k^{9} \log^{3} k) (58.72 + 9 \log k + 3 \log (\log k))^{2} \\
& < 5.1 \times 10^{29} k^{9} \log^{5} k,
\end{align*}
where we have used $58.72 + 9 \log k + 3 \log (\log k) < 62.8 \log k $ for $k \geq 3$. \\
The result established in this subsection is summarized in the following lemma.
\begin{lemma}\label{lem3.1}
If $(a, b, k, l, m, n)$ is a solution of equation \eqref{eq 1.2} with $k \geq 3$ and $n \geq k+2$, then the inequalities 
\begin{equation}\label{eq 3.34}
 n < 5.1 \times 10^{29} k^{9} \log^{5} k
\end{equation}
holds.
\end{lemma}

\subsection{The case of small $k$}
In this subsection, we consider $k \in [3, 640]$. Let us define
 \begin{equation}\label{eq 3.35}
\Gamma_{1} := \log(\Lambda_{1} + 1) = n \log \gamma -(m+l) \log 10 + \log \left(81(2\gamma - 2) g_{k}(\gamma) / ab \right).     
\end{equation}
Assume that $l \geq 2$, so it comes from \eqref{eq 3.26} that $|\Lambda_{1}| < 0.183$. Now by Lemma \ref{lem 2.5}, we have 
\begin{equation}\label{eq 3.36}
 |\Gamma_{1}| < \frac{- \log 0.817}{0.183}\cdot |\Lambda_{1}|  < 20.22 \cdot \exp (-2.3 \cdot l). 
\end{equation}
In order to apply Lemma \ref{lem 2.4} to $\Gamma_{1}$, set
\[
c := 20.22, \quad \rho := 2.3 , \quad \psi := \frac{\log \left(81 (2\gamma-2)g_{k}(\gamma)/ ab \right)}{\log 10},
\]
and
\[ \vartheta:= \frac{\log \gamma}{\log 10}, \quad \vartheta_{1}:=  \log \gamma, \quad \vartheta_{2}:= - \log 10, \quad \delta:= \log \left(81 (2\gamma-2)g_{k}(\gamma)/ ab \right).
\]
From Lemma \ref{lem3.1}, we take $X_{0} = \left\lfloor 5.1 \times 10^{29} k^{9} \log^{5} k \right \rfloor$ which is an upper bound of $\max \{ n, m+l \}$. For any $k \in [3, 640]$, we find a good approximation of $\gamma$ and a convergent $p_{t}/q_{t}$ of the continued fraction of $\vartheta$ such that $q_{t} > X_{0}$. After that, we apply Lemma \ref{lem 2.4} to inequality \eqref{eq 3.35}. Using \textit{Mathematica} we found  that if $k \in [3,640]$, then the maximum value of $\left\lfloor \frac{1}{\rho} \log \left( q^{2} c / |\vartheta_{2}| X_{0}\right) \right \rfloor$ is 117.892, which is an upper bound on $l$ according to Lemma \ref{lem 2.4}.

Now, we fix $1 \leq l \leq 117$ and consider 
\begin{equation}\label{eq 3.37}
\Gamma_{2} := \log(\Lambda_{2} + 1) = n \log \gamma - m \log 10 + \log \left(81(2\gamma - 2) g_{k}(\gamma) / (ab(10^{l} -1)) \right).   \end{equation}
Assume that $m \geq 2$. Thus from \eqref{eq 3.32}, we get $|\Lambda_{2}| < 0.19$. Then by Lemma \ref{lem 2.5}, we have 
\begin{equation}\label{eq 3.38}
 |\Gamma_{2}| < \frac{- \log 0.81}{0.19}\cdot |\Lambda_{2}|  < 21.1 \cdot \exp (-2.3 \cdot m). 
\end{equation}
Now to apply Lemma \ref{lem 2.4} to $\Gamma_{2}$, set
\[
c := 21.1, \quad \rho := 2.3, \quad \vartheta:= \frac{\log \gamma}{\log 10}, \quad \vartheta_{1}:=  \log \gamma, \quad \vartheta_{2}:= - \log 10,
\]
and
\[ \psi_{l} := \frac{\log \left(81 (2\gamma-2)g_{k}(\gamma)/ (ab(10^{l} -1)) \right)}{\log 10}, \quad \delta_{l}:= \log \left(81 (2\gamma-2)g_{k}(\gamma)/ (ab(10^{l} -1)) \right).
\]
Again, for any $(k, l) \in [3, 640] \times [1,117]$, we find a good approximation of $\gamma$ and a convergent $p_{t}/q_{t}$ of the continued fraction of $\vartheta$ such that $q_{t} > X_{0}$, where $X_{0} = \left\lfloor 5.1 \times 10^{29} k^{9} \log^{5} k \right \rfloor$ which is an upper bound of $\max \{ n, m \}$. After that, we apply Lemma \ref{lem 2.4} to inequality \eqref{eq 3.38}.  Using \textit{Mathematica} we found  that for $k \in  (k, l) \in [3,640] \times [1,117]$, then the maximum value of $\left\lfloor \frac{1}{\rho} \log \left( q^{2} c / |\vartheta_{2}| X_{0}\right) \right \rfloor$ is 117.91, which is an upper bound on $m$ according to Lemma \ref{lem 2.4}.

Thus, we derive that all the possible solutions $(a, b, k, l, m, n)$ of the equation \eqref{eq 1.2} for which $k \in [3, 640]$ have $1 \leq l \leq m \leq 117$. Hence we apply inequalities \eqref{eq 3.24} to find $n \leq 782$.

Finally, we used \textit{Mathematica} to  compare $Q_{n}^{(k)}$ and $\frac{a(10^{l}-1)}{9} \cdot \frac{b(10^{m}-1)}{9}$ for the range $5 \leq n \leq 782$ and $1 \leq l \leq m \leq 117$, with $ m+l < 0.42n + 2.31$ and checked that the only solution of the equation \eqref{eq 1.2} are $Q_{5}^{(k)}$ and $Q_{7}^{(4)}$.

\subsection{The case of large $k$}
In this subsection, we examine the case where $k > 640$.
\subsection{An absolute upper bound on $k$}
For $k > 640$, the following inequalities hold
\[
n < 5.1 \times 10^{29} k^{9} \log^{5} k < \phi^{k/2}.
\]
From Lemma \ref{lem 2.9}, we have
\begin{equation}\label{eq 3.39}
(2 \gamma -2 )g_{k}(\gamma)\gamma^{n} = \frac{2 \phi^{2n+1}}{\phi+2}(1+ \xi),  \quad \text{where} \quad |\xi| < \frac{1.25}{\phi^{k/2}}.   
\end{equation}
Substituting \eqref{eq 3.39} into \eqref{eq 3.25}, provides us the inequality
\begin{align*}
    \left|\frac{ab 10^{l+m}}{81} - \frac{2 \phi^{2n+1}}{\phi+2} \right| & \leq \left|\frac{ab 10^{l+m}}{81} - g_{k}(\gamma) \gamma^{n} \right| + \frac{2 \phi^{2n+1}}{\phi+2}|\xi| \\
    & \leq \frac{ab}{81}(10^{l} +10^{m} +1 ) + 2+ \frac{2 \phi^{2n+1}}{\phi+2} \cdot \frac{1.25}{\phi^{k/2}} \\
    & \leq  2 \cdot 10^{m} + 3 + \frac{2 \phi^{2n+1}}{\phi+2} \cdot \frac{1.25}{\phi^{k/2}}.
\end{align*}
Dividing both sides of the above inequality  by $\frac{2 \phi^{2n+1}}{\phi+2}$ and using the facts $10^{m+l-2} < 2\gamma^{n} < 2 \phi^{2n}$ together with $n \geq k + 2$, it yields
\begin{equation}\label{eq 3.40}
  \left|\frac{ab 10^{l+m}}{162}(\phi+2) \phi^{-(2n+1)} - 1 \right| < \frac{2(\phi +2)}{\phi \cdot 10^{l-2}} + \frac{3(\phi +2)/2 \phi +1.25}{\phi^{k/2}} 
  < \frac{452}{\phi^{\lambda}},
\end{equation}
where $\lambda := \min \{k/2, \theta l\}$ and $\theta := \log 10 / \log \phi$.
\\
In order to apply Theorem \ref{thm2} in \eqref{eq 3.40}, set 
\[
\Lambda_{3} := \frac{ab 10^{l+m}}{162}(\phi+2) \phi^{-(2n+1)} - 1.
\]
If  $\Lambda_{3} = 0$, then 
\[
\frac{ab 10^{l+m}}{162}(\phi+2) = \phi^{2n+1}.
\]
Conjugating the above relation in $\mathbb{Q}(\sqrt{5})$, it gives
\[
4 < \frac{ab 10^{l+m}}{162}(\overline{\phi}+2) = \overline{\phi}^{2n+1} < 1,
\]
where $\overline{\phi} = (1 - \sqrt{5})/2$. Therefore, $\Lambda_{3} \neq 0$. We take $s := 3$,
\[
 \eta_{1}:= \frac{ab(\phi+2)}{162} , \quad  \eta_{2}:= \phi, \quad  \eta_{3}:= 10,
\]
and
\[ 
a_{1}:= 1, \quad a_{2}:= -(2n+1), \quad a_{3}:= m+l.
\]
Note that $\mathbb{L} := \mathbb{Q}(\sqrt{5})$ contains $\eta_{1}, \eta_{2}, \eta_{3}$ and has $d_{\mathbb{L}}:= 2$. Since $m+l < 2n+1$, we deduce that $D:= \max\{|a_{1}|,|a_{2}|,|a_{3}|\}= 2n+1$. Moreover, since $h(\eta_{2}) = \frac{\log \phi}{2}$, $h(\eta_{3}) = \log 10$ and
\begin{align*}
 h(\eta_{1}) & \leq  h(a) + h(b) + h(162) + h(\phi)+ h(2) + \log 2 \\
 & \leq 4 \log 9 +3 \log 2 + \frac{\log \phi}{2} < 11.2.
 \end{align*}
Therefore, we may take
\[
B_{1} := 22.4, \quad B_{2} := \log \phi \quad \text{and} \quad B_{3} := 2 \log 10.
\]
As before, applying Theorem \ref{thm2}, we have 
\begin{equation}\label{eq 3.41}
- \log |\Lambda_{3}| < 1.1 \times 10^{14} \log n,   
\end{equation}
where $1 + \log (2n+1) < 2.2 \log n $ holds for all $n \geq 5$. Comparing \eqref{eq 3.40} and \eqref{eq 3.41}, it yields
\begin{equation}\label{eq 3.42}
\lambda < 2.3 \times 10^{14} \log n.    
\end{equation}
Now, we distinguish two cases according to $\lambda$. \\
\textbf{Case 1} $\lambda = k/2$. According to Lemma \ref{lem3.1} and inequality \eqref{eq 3.32}, it yields
\[
k < 4.6 \times 10^{14} \log \left( 5.1 \times 10^{29} k^{9} \log^{5} k\right).
\]
Solving the above inequality and putting the result in the inequality of Lemma \ref{lem3.1}, it yields
\begin{equation}\label{eq 3.43}
k < 7.15 \times 10^{17} \quad \text{and} \quad n < 2.93 \times 10^{198}.     
\end{equation}
\textbf{Case 2} $\lambda = \theta l$. It follows in this case from \eqref{eq 3.32} that
\begin{equation}\label{eq 3.44}
    l < 4.81 \times 10^{13} \log n.
\end{equation}
Returning to equation \eqref{eq 3.30} and using \eqref{eq 3.39} in it, we obtain
\begin{align*}
    \left | \frac{b\cdot 10^{m}}{9} - \frac{18 \cdot \phi^{2n +1}}{a(10^{l}-1)(\phi + 2)} \right| & \leq \left| \frac{b\cdot 10^{m}}{9} - \frac{9(2\gamma - 2)g_{k}(\gamma)}{a(10^{l}-1)} \right| + \frac{18 \phi^{2n+1}}{a(10^{l}-1)(\phi + 2)} |\xi| \\
    & \leq \frac{18}{a(10^{l}-1)} + 1 + \frac{18 \phi^{2n+1}}{a(10^{l}-1)(\phi + 2)} \cdot \frac{1.25}{\phi^{k/2}},
\end{align*}
which implies that 
\begin{equation}\label{eq 3.45}
  \left | \frac{ab(10^{l}-1)}{162} (\phi +2) \phi^{-(2n+1)} 10^{m} -1 \right|   \leq \frac{\phi +2}{\phi^{2n+1}} + \frac{9 (\phi +2)}{2 \phi^{2n+1}} + \frac{1.25}{\phi^{k/2}} < \frac{22}{\phi^{k/2}},
\end{equation}
where $n \geq k+2$ and $\theta l < \frac{k}{2}$. Applying Theorem \ref{thm2} to \eqref{eq 3.45} with the data $s:= 3$ and 
\[
\Lambda_{4} := \frac{ab(10^{l}-1)(\phi +2)}{162}  \phi^{-(2n+1)} 10^{m} - 1,
\]
where 
\[
 \eta_{1}:= \frac{ab(10^{l}-1)(\phi+2)}{162} , \quad  \eta_{2}:= \phi, \quad  \eta_{3}:= 10,
\]
and
\[ 
a_{1}:= 1, \quad a_{2}:= -(2n+1), \quad a_{3}:= m .
\]
 Here we can show that $\Lambda_{4} \neq 0$ by using the same method as before to prove $\Lambda_{3} \neq 0$. As previously calculated, we can take
 \[
 d_{\mathbb{L}}:= 2, \quad B_{2}:= \log \phi, \quad \quad B_{3}:=  2 \log 10  \quad  and \quad D:= 2n+1.
 \]
Moreover, by using \eqref{eq 3.44}, we have 
\begin{align*}
h(\eta_{1}) & \leq h(a) + h(b) + h(162)  + h(10^{l}-1) + h(\phi +2)  \\ 
& \leq 4 \log 9 + 4 \log 2+ l \log 10 + \frac{\log \phi}{2} \\
& \leq 4 \log 9 + 4 \log 2 + 4.81 \times 10^{13} \log n \times \log 10 + \frac{\log \phi}{2} \\
& \leq 11.1 \times 10^{13} \log n. 
\end{align*}
Hence, we can take $B_{1} := 22.2 \times 10^{13} \log n.$ According to Theorem \ref{thm2} and inequality \eqref{eq 3.45}, we obtain 
\[
\exp{(-1.1 \times 10^{27} (\log n)^{2})} < \frac{22}{\phi^{k/2}}.
\]
As a result, we get 
\begin{equation}\label{eq 3.46}
k < 4.58 \times 10^{27} (\log n)^{2}.    
\end{equation}
From inequality \eqref{eq 3.46} and Lemma \ref{lem3.1}, we have
\begin{equation*}
    k < 4.1 \times 10^{34} \quad \text{and} \quad  n < 5.37 \times 10^{350}.
\end{equation*}
The result established in this subsection is summarized in the following lemma.
\begin{lemma}\label{lem3.2}
If $(a, b, k, l, m, n)$ is a solution of equation \eqref{eq 1.2} with $k > 640$ and $n \geq k+2$, then $k$ and $n$ are bounded as
\begin{equation}\label{eq 3.47}
k < 4.1 \times 10^{34} \quad \text{and} \quad  n < 5.37 \times 10^{350}.
\end{equation}
\end{lemma}
\subsection{Reducing the bound on $k$}
Let us define
\[ 
\Gamma_{3} := \log(\Lambda_{3} + 1) = (m+l) \log 10 -(2n+1) \log \phi - \log \left(ab(\phi  + 2) / 162 \right).     
\]
Assume that $l \geq 3$, so it comes from \eqref{eq 3.40} that $|\Lambda_{3}| < 0.46$. Now by Lemma \ref{lem 2.5}, we have 
\begin{equation}\label{eq 3.48}
 |\Gamma_{3}| < \frac{- \log (0.54)}{0.46}\cdot |\Lambda_{3}|  < 606 \cdot \exp (-0.48 \cdot \lambda). 
\end{equation}
For $a,b \in \{1,2, \cdots ,9 \}$, we apply Lemma \ref{lem 2.4} to $\Gamma_{3}$ with the parameters
\[
c := 606, \quad \rho := 0.48 , \quad \psi := \frac{\log \left(ab (\phi + 2)/ 162 \right)}{\log 10},
\]
and
\[ \vartheta:= \frac{\log \phi}{\log 10}, \quad \vartheta_{1}:= - \log \phi, \quad \vartheta_{2}:=  \log 10, \quad \delta:= \log \left(ab (\phi + 2)/ 162 \right).
\]
Further $m+l < n$ and $n < 5.37 \times 10^{350}$ by Lemma \ref{lem3.2},  then we can take $X_{0} = 5.37 \times 10^{350}$. Let 
\[
 [a_{0}, a_{1}, a_{2},\dots] = [0, 4, 1, 3, 1, 1, 1, 6, 4, 2, 1, 10, 1, 4, 46, 3,\dots]
 \]
 be the continued fraction expansion of $\vartheta$ . With the help of \textit{Mathematica}, we find that 
$q_{676} \approx 1.91 \times 10^{355}$ satisfies the conditions of Lemma \ref{lem 2.4}. Furthermore, according
to Lemma \ref{lem 2.4} we obtain $\lambda < 1738$.
\\
\textbf{Case 1} $\lambda = k/2$. In this case we obtain that
\begin{equation}\label{eq 3.49}
    k \leq 3476.
\end{equation}
\textbf{Case 2} $\lambda = \theta l$. In this case we get that $l \leq 363$. Now let 
\[
\Gamma_{4} := \log(\Lambda_{4} + 1) = m \log 10 -(2n+1) \log \phi + \log \left(ab(10^{l}-1)(\phi  + 2) / 162 \right).
\]
Since $k > 640$, then from \eqref{eq 3.45}, we have  $|\Lambda_{4}| < 0.01$. Hence by using Lemma \ref{lem 2.5}, we find that
\begin{equation}\label{eq 3.50}
 |\Gamma_{4}| < \frac{- \log (0.99)}{0.01}\cdot |\Gamma_{4}|  < 22.12  \cdot \exp (-0.24 \cdot k ). 
\end{equation}
For $a,b \in \{1,2, \cdots ,9 \}$ and $1 \leq l \leq 363$, we apply Lemma \ref{lem 2.4} to $\Gamma_{4}$ with the parameters
\[
c := 22.12, \quad \rho := 0.24 , \quad \psi := \frac{\log \left(ab(10^{l} - 1) (\phi + 2)/ 162 \right)}{\log 10}, \quad 1 \leq l \leq 363.
\]
and
\[ \vartheta:= \frac{\log \phi}{\log 10}, \quad \vartheta_{1}:= - \log \phi, \quad \vartheta_{2}:=  \log 10, \quad \delta:= \log \left(ab (10^{l} - 1) (\phi + 2)/ 162 \right).
\]
Further, Lemma \ref{lem3.2} implies that we can take $X_{0} = 5.34 \times 10^{350}$. Using
\textit{Mathematica}, we find that $q_{680} \approx 2.12 \times 10^{359}$ satisfies the conditions of Lemma \ref{lem 2.4}. Furthermore, according to Lemma \ref{lem 2.4} we obtain $k \leq 3494$. With this new upper bound on $k$, we get
\[
n < 1.44 \times 10^{66}.
\]
We now repeat the above process  with $X_{0} = 1.44 \times 10^{66}$, we obtain that
$k \leq 683 $. Hence, we deduce that
\[
n < 1.96 \times 10^{59}.
\]
In the third application with $X_{0} = 1.96 \times 10^{59}$, we get that $k < 634$, which contradicts our assumption $k > 640$. This completes the proof of Theorem \ref{thm1}.

\vspace{05mm} \noindent \footnotesize
\begin{minipage}[b]{90cm}
\large{School of Applied Sciences, \\ 
KIIT University, Bhubaneswar, \\ 
Bhubaneswar 751024, Odisha, India. \\
Email: bijan.patelfma@kiit.ac.in}
\end{minipage}

\vspace{05mm} \noindent \footnotesize
\begin{minipage}[b]{90cm}
\large{School of Applied Sciences, \\ 
KIIT University, Bhubaneswar, \\ 
Bhubaneswar 751024, Odisha, India. \\
Email: bptbibhu@gmail.com}
\end{minipage}

\end{document}